\documentclass[12pt]{article}
\usepackage{amsmath, amsfonts, amsthm, amssymb, epsfig}

\newtheorem{theorem}{Theorem}
\newtheorem{lemma}{Lemma}
\newtheorem{result}{Result}

\begin{document}

\title{{\bf Five Exponential Diophantine Equations and Mayhem Problem M429}}
\author{Konstantine Zelator\\
Department of Mathematics, Computer Science and Statistics\\
212 Ben Franklin Hall\\
Bloomsburg University of Pennsylvania\\
400 East Second Street\\
Bloomsburg, PA  17815\\
USA\\
and\\
P.O. Box 4280\\
Pittsburgh, PA  15203\\
e-mails: konstantine\underline{\ }zelator@yahoo.com\\
and kzelator@bloomu.edu}
\maketitle

\section{Introduction}

In the March 2010 issue of the journal {\it Crux Mathematicorum with
  Mathematical Mayhem}, mayhem  problem M429 was proposed (see reference \cite{1}):

\vspace{.15in}

Determine all positive integers $a,b,c$ that satisfy,

$$
\begin{array}{rcll} a^{(b^c)} & = &  (a^b)^c; & {\rm or\ equivalently}\\
\\
a^{b^c} & = & a^{bc}.
\end{array}
$$

\noindent A solution, by this author, was published in the December 2010
issue of {\it Crux Mathematicorum with Mathematical Mayhem} (see
\cite{2}).  According to this solution, the following ordered triples of
positive integers $a,b,c$ are precisely those that satisfy the above exponentialequation:

\vspace{.15in}

\noindent The triples of the form $(1,b,c)$, with $b,c$ being any positive
integers;

\noindent the triples of the form $(a,b,1)$, with $a,b$ positive integers and
with $a \geq 2$;

\noindent and the triples of the form $(a,2,2)$ with $a \in {\Bbb Z}^+$, and $a
\geq 2$.

\vspace{.15in}

In the language of diophantine equations, we are dealing with the
three-variable diophantine equation

\begin{equation}
x^{(y^z)} = x^{yz}. \label{E1}
\end{equation}

\noindent Accordingly, the above results can be expressed in Theorem 1 as
follows.

\begin{theorem}  Consider the three-variable diophantine equation, $x^{(y^z)} =
  x^{yz}$, over the set of positive integers ${\Bbb Z}^+$.  If $S$ is the
  solution set of the above diophantine equation, then $S = S_1 \bigcup S_2
  \bigcup S_3$, where $S_1,S_2,S_3$ are the pairwise disjoint sets,

$$
\begin{array}{rcl} S_1 & = & \left\{ \left.(1,b,c)\right| b,c \in {\Bbb Z}^+
  \right\};\\
\\
S_2 & = & \left\{\left. (a,b,1)\right| a \geq 2, a,b \in {\Bbb Z}^+\right\};\\
\\
S_3 & = & \left\{ \left.(a,2,2)\right| a \geq 2\ {\rm and}\ a \in {\Bbb
  Z}^+\right\}.
\end{array}
$$
\end{theorem}

Motivated by mayhem problem M429, in this work we tackle another four
exponential, three-variable diophantine equations.  These are:

\vspace{.15in}

\begin{equation}
x^{(y^z)} = x^{(z^y)}, \label{E2}
\end{equation}

\begin{equation} 
x^{(y^z)} = y^{xz}, \label{E3}
\end{equation}

\begin{equation}
x^{yz} = y^{xz}, \label{E4}
\end{equation}

\noindent and

\begin{equation}
x^{(y^z)} = z^{xy} \label{E5}
\end{equation}

In Section 2, we state Theorems 2, 3, 4, and 5.  In Theorems 2, 3 and 4, the
solutions sets of the diophantine equations (\ref{E2}), (\ref{E3}), and
(\ref{E4}) are stated.

These three solution sets are determined with the aid of the two-variable
exponential diophantine equation found in Section 3, whose solution set is
given in Result 2.

The proofs of Theorems 2,3, and 4, are given in Section 4.  The proof of
Theorem 5 is presented in Section 5.  In Theorem 5, some solutions to equation
(\ref{E5}) are given.

\section{The four theorems}

\begin{theorem} Consider the three-variable diophantine equation (over ${\Bbb
    Z}^+$),

$$ x^{(y^z)}  = x^{z^y}.
$$

\noindent Let $S$ be the solution set of this equation.

Then, $S = S_1 \bigcup S_2 \bigcup S_3 \bigcup S_4 \bigcup S_5$, where 

\vspace{.15in}

$
\begin{array}{rcll}
S_1 & = & \left\{\left. (1,b,c)\right| b,c \in {\Bbb Z}^+ \right\}& {\rm where}\\
\\
S_2 & = & \left\{ \left. (a,1,1)\right|  a\geq 2, a \in {\Bbb Z}^+ \right\} & \\
\\
S_3 & = & \left\{\left. (a,b,b) \right| a \geq 2, b\geq 2, a, b \in {\Bbb
  Z}^+\right\}&\\
\\
S_4 & = & \left\{ \left. (a,4,2)\right| a \geq 2,\ a \in {\Bbb Z}^+\right\}&\\
\\
S_5 & = & \left\{ \left. (a,2,4)\right| a \geq 2, a \in {\Bbb Z}^+ \right\} & 
\end{array}
$
\end{theorem}

\vspace{.15in}

\begin{theorem} Consider the three-variable diophantine equation (over ${\Bbb
    Z}^+$),

$$
x^{(y^z)} = y^{xz}
$$

Let $S$ be the solution set of this equation.  Then, $S = S_1 \bigcup S_2
\bigcup S_3 \bigcup S_4 \bigcup S_5$ where

$$\begin{array}{lrcl}
&S_1&  =&  \left\{ \left.(1,1,c)\right| c \in {\Bbb Z}^+\right\}, \\
\\
&S_2 & = & \left\{ \left. (a,a,1)\right| a \geq  2, a \in {\Bbb Z}^+\right\} \\
\\
{\rm (singleton\ set)} & S_3 & =&  \left\{( 4,2,1)\right\},\\
\\
{\rm (singleton\ set)} & S_4 & = & \left\{(2,4,1)\right\} \\
\\
& S_5 & = & \left\{ \left. (b^c,b,c)\right|b \geq 2, c \geq 2, b,c\in {\Bbb
  Z}^+\right\}
\end{array}
$$
\end{theorem}

\begin{theorem} Consider the three-variable diophantine equation (over ${\Bbb
    Z^+} $

$$
x^{yz} = y^{xz}.
$$

Let $S$ be its solution set.  Then,

$$
 S  =  S_1 \bigcup S_2 \bigcup S_3 \bigcup S_4 \bigcup S_t \bigcup S_6
  \bigcup S_7,
$$

\noindent where

$$
\begin{array}{lrcl}
& S_1 & = & \left\{ \left. (1,1,c)\right| c \in {\Bbb Z}^+\right\} \\
\\
& S_2 & = & \left\{ \left. (a,a,1)\right| a \geq 2, a \in {\Bbb Z}^+\right\}\\
\\
{\rm (singleton \ set)} & S_3 & = & \left\{ (4,2,1)\right\},\\
\\
{\rm (singleton \ set)} & S_4 & = & \left\{ (2,4,1)\right\}\\
\\
& S_5 & = & \left\{\left. (a,a,c)\right| a\geq 2, c\geq 2, a,c \in {\Bbb
  Z}^+\right\}\\
\\
& S_6 & = & \left\{ \left. (4,2,c)\right| c \geq 2,\ c\in {\Bbb Z}^+\right\}
\\
\\
& S_7 & = & \left\{ \left. (2,4,c) \right| c \geq , \ c \in {\Bbb Z}^+\right\}
\end{array}
$$

\end{theorem}

\vspace{.15in}

\begin{theorem} Consider the  three-variable equation (over ${\Bbb Z}^+$)

$$
x^{(y^z)} = z^{xy}
$$

\begin{enumerate}
\item[(i)]  Let $S$ be the set of those solutions, $(x,y,z)$ such that at
  least one of $x,y$, or $z$ is equal to $1$.  Then
 
$$
S = \left\{ \left. (1,b,1)\right| b \in {\Bbb Z}^+\right\}
$$

\item[(ii)] The only solution $(x,y,z)$ to the above equation, such that $x
  \geq 2,\ y \geq 2,\ z \geq 2$, and with $x=z$, is the triple $(2,2,2)$

\item[(iii)]  Let $F$ be the family of solutions $(x,y,z)$ such that $x \geq
  2,\ y \geq 2,\ z\geq 2$ and with $y = z \neq x$.  Then

$$
F=\left\{ \left.(b^b,b,b) \right| b \geq 2,\ b \in {\Bbb Z}^+\right\}
$$
\end{enumerate}
\end{theorem}

\section{A key exponential diophantine equation}

The diophantine equation, $x^y = y^x$, over the positive integers, is
instrumental in determining the solution sets of the diophantine equations
(\ref{E2}), (\ref{E3}), and (\ref{E4}).  The following, Result 1, can be found
in W. Sierpinski's book, ``Elementary Theory of Numbers'', (see reference
\cite{3}).  The proof is about half a page long.

\begin{result} Consider the two-variable equation, $x^y = y^x$, over the set
  of positive rational numbers, ${\Bbb Q}^+$.  Then all the solutions to this
  equation, with $x$ and $y$ being positive rationals, and with $y > x$, are
  given by

$$
x = \left( 1+\dfrac{1}{n}\right)^n,\ \ \ y= \left(
1+\dfrac{1}{n}\right)^{n+1},
$$

\noindent where $n$ is a positive integer: $n = 1,2,3,\ldots$\ .
\end{result}

A simple or cursory examination of the formulas in Result 1 easily leads to
Result 2.  Observe that these formulas can be written in the form,

$$
x=\left(\dfrac{n+1}{n}\right)^n ,\ \ \ y = \left(\dfrac{n+1}{n} \right)^{n+1}.
$$

For $n = 1$, we obtain the integer solution $x = 2$ and $y = 4$.  However,
for $n \geq 2$, the number $\dfrac{n+1}{n}$ is a proper rational, i.e., a
rational which is not an integer.

This is clear since $n$ and $n+1$ are relatively prime, and $n \geq 2$.  Thus,
since for $n \geq 2$, $\dfrac{n+1}{n}$ is a proper rational, so must be any
positive integer power of $\dfrac{n+1}{n}$.  This observation takes us
immediately to Result 2 below.

\begin{result}  Consider the two-variable diophantine equation (over ${\Bbb
    Z}^+$)

$$
x^y = y^x.
$$

Let $S$ be its solution set.  Then, $S = S_1 \bigcup S_2 \bigcup S_3$.  Where

$$
\begin{array}{lrcl} & S_1 & = & \left\{ \left. (a,a)\right| a \in {\Bbb Z}^+
  \right\},\\
\\
{\rm (singleton\ set)} & S_2 & = & \left\{ (4,2)\right\},\\
\\
{\rm and\ (singleton\ set)} & S_3 & = & \left\{ (2,4)\right\}
\end{array}
$$
\end{result} 

Result 2 is used in the proofs of Theorems  2, 3, and 4 below.

\section{Proofs of Theorems 2, 3, and 4}

\begin{enumerate}\item[(1)] \begin{proof} {\bf Theorem 2} Suppose that $(a,b,c)$  is a solution to equation (\ref{E2}).  We have 

\begin{equation}
a^{(b^c)} = a^{(c^b)} \label{E6}
\end{equation}

\noindent If $a=1$, then $b$ and $c$ can be arbitrary positive integers; and
(\ref{E6}) is satisfied.

\noindent If $b=1$ and $a \geq 2$, then by (\ref{E6}) we get 

\vspace{.15in}

\hspace{2.25in} $a=a^c$.  \hfill (6a) 

\vspace{.15in}

\noindent Since $ a \geq 2$, by inspection, we see that (6a) is satisfied only
when $c = 1$.

So, we obtain the solutions of the form $(a,1,1)$ with $a \geq 2$.  If $a \geq
2, \ b\geq 2$, and $c =1$, equation (\ref{E6}) yields

$$
a^b = a,
$$

\noindent which is impossible with $a \geq 2$ and $b \geq 2$.

Finally, assume that $a \geq 2,\ b \geq 2$, and $c\geq 2$ in (\ref{E6}).  Then
(\ref{E6}) $\Leftrightarrow$ (since $a \geq 2$) $b^c = c^b$; and by Result 2,
it follows that either $b=4$ and $c=2$; or $b=2$ and $c =4$; or $b = c$.  We
have shown that if $(a,b,c)$ is a positive integer solution of equation
(\ref{E2}), then $(a,b,c)$ must belong to one of the sets $S_1,S_2,S_3,S_4$, or $S_5$.
Conversely, a routine calculation shows that any member of these five sets is
a solution to (\ref{E2}).
\end{proof}
\item[(2)] \begin{proof} {\bf Theorem 3}. Let $(a,b,c)$ be a solution to equation (\ref{E3}).  We then have,

\begin{equation}
a^{(b^c)} = b^{ac} \label{E7}
\end{equation}

\noindent If $a =1$, then by (\ref{E7}), $1=b^c$, which in turn  implies
$b=1$; and $c$ an arbitrary positive integer.  

\noindent If $a \geq 2$ and $b=1$, (\ref{E7}) becomes impossible for any value
of $c$.  If $a \geq 2,\ b\geq 2$, and $c = 1$, (\ref{E7}) yields $a^b = b^a$;
and by Result 2 we must have either $a = 4$ and $b=2$, or $a = 2$ and $b=4$;
or $a =b$.  If $a \geq 2,\ b\geq 2,\ c \geq 2$.  Then by (\ref{E7}),

\vspace{.15in}

\hspace{2.25in} $a^{(b^c)} = (b^c)^a$ \hfill (7a)

\vspace{.15in}

\noindent Combining (7a) with Result 2 implies that either $a=4$ and $b^c=2$,
which is impossible since $b \geq 2$ and $c \geq 2$, or that $a=2$ and $b^c =
4$, which gives $a = 2 = b = c$.  Or, the third possibility, $a = b^c$.  We
have shown that if $(a,b,c)$ is a positive integer solution of equation
(\ref{E3}), then it must belong to one of the sets $S_1,\ S_2,\ S_3,\ S_4$ or
$S_5$. Conversely, a routine calculation shows that any member of these five
sets is a solution to (\ref{E3}). 
\end{proof}

\item[(3)] \begin{proof} {\bf Theorem 4}.  Let $(a,b,c)$ be a positive integer
  solution to equation (\ref{E4})

\begin{equation}
a^{bc} = b^{ac} \label{E8}
\end{equation}

\noindent If $a=1$, we obtain $1=b^c$; and so $b=1$, with $c$ being an
arbitrary positive integer.

\noindent If $a \geq 2$ and $b=1$, (\ref{E8}) gives $a^c=1$, which is
impossible since $a \geq 2$.

\noindent If $a \geq 2,\ b \geq 2$, and $c=1$, we obtain from (\ref{E8}) 

\vspace{.15in}

\hspace{1.0in} $a^b = b^a$ \hfill (8a)

\vspace{.15in}

\noindent Equation (8a), combined with Result 2, implies that either $a = 4$ and
$b=2$; or $a = 2$ and $b=4$; or $a = b$.  If $a \geq 2,\ b\geq 2$, and $c \geq
2$, we have from (\ref{E8})

\vspace{.15in}

\hspace{1.0in} $a^{bc} = b^{ac} \Leftrightarrow (a^b)^c = (b^a)^c$ \hfill (8b)

\vspace{.15in}

\noindent Equation (8b) demonstrates that the $c$th powers of the positive
integers $a^b$ and $b^a$ are equal.  Since these two integers are greater than
$1$, equation (8b) implies

$$a^b=b^a$$

\noindent which once more, when combined with Result 2, implies either $a = 4$
and $b=2$ or $a=2$ and $b=4$; or $a =b$.  

We have shown that if $(a,b,c)$ is a positive integer solution of equation
(\ref{E4}), it must belong to one of the sets $S_1,\ S_2,\ S_3,\ S_4,\ S_5, \
S_6,$ or $S_7$.  Conversely, a routine calculation establishes  that any
member of these seven sets is a solution to (\ref{E4}).
\end{proof}

\item[(5)] {\bf Proof of Theorem 5}  The following lemma can be easily
  proved by using mathematical induction.  We omit the details.  We will use
  the lemma in the proof of Theorem 5.

\begin{lemma} 
\ \ \ \ 
\begin{enumerate} 
\item[(i)]  If $b \geq 3$, then $b^{n-1} > n$ for all positive integers $n
  \geq 2$.

\item[(ii))] $2^{n-1} > n$, for all positive integers $n \geq 3$.

\item[(iii)]  If $c \geq 2$, then $c^n > n$, for all positive integers $n$.
  
\end{enumerate}
\end{lemma}

\begin{proof} {\bf Theorem 5} 
\ \ \ \ 
\begin{enumerate}
\item[(i)] Let $(a,b,c)$ be a solution to equation (\ref{E5}) with at least
  one of $a,b,c$ being equal to $1$.

\vspace{.15in}

\noindent If $a=1$, (\ref{E5}) implies $1 = c^b$, and so $c=1$ as well; and
$b$ is an arbitrary positive integer.

\noindent If $b=1$ and $a \geq 2$ we get $a =c^a$ which is impossible if $c
\geq 2$, by Lemma  1(iii); and clearly,  $c \neq 1$ since $a \geq 2$.

Also, the case $b \geq 2, \ a \geq 2$, and $c=1$ is ruled out by inspection.
We conclude that if $(a,b,c)$ is a solution to (\ref{E5}), with one of $a,b,c$
being $1$, then it must be of the form $(1,b,1)$.  Conversely, a
straightforward calculation established that $(1,b,1)$ is a solution of
(\ref{E5}) for every positive integer $b$.

\item[(ii)]  Let $(a,b,c)$ be a solution to (\ref{E5}) with $a \geq 2,\ b\geq
  2,\ c\geq 2$, and $a =c$.  We have, by (\ref{E5}), $a^{(b^a)} = a^{ab}
  \Leftrightarrow $ (since $a\geq 2$) $b^a = ab$, or equivalently, $b^{a-1} =
  a$, which, when combined with Lemma 1, parts (i) and (ii), implies that
  either $b \geq 3$ and $a=1$; which is ruled out since $a \geq 2$; or
  alternatively, $b = 2$ and $a \leq 2$ which gives $a = 2$.  We obtain $a = b
  = c = 2$.  Conversely, $(2,2,2)$ is a solution to equation (\ref{E5}), $2^4
  = 2^4$.

\item[(iii)]  Let $(a,b,c)$ be a solution to (\ref{E5}) with $a \geq 2,\ b
  \geq 2,\ c \geq 2$, and with $b = c \neq a$.  We have,

$$ a^{(b^b)} = b^{ab};
$$

\noindent  or equivalently,

\begin{equation} a^{(b^b)}  =  (b^{b})^a\label{E9}
\end{equation}

Equation (\ref{E9}) combined with Result 2 implies that, either $a = 4$ and
$b^b = 2$ or $a = 2$ and $b^b=4$; or $a = b^b$.  The first possibility is
ruled out since $b^b \geq 2^2 > 2$, because $b \geq 2$ .  The second
possibility yields $b^b=4,\ b=2$, but then also have $a=2$ and so $a = b=c=2$,
contrary to $b = c \neq a$.  The third possibility establishes $(a,b,c) =
(b^b,b,b)$.  Conversely, $(b^b,b,b)$ is a solution to (\ref{E5}) for any
positive integer $b \geq 2$.  Both sides of (\ref{E5}) are equal to
$b^{(b^{b+1})}$. 
\end{enumerate}
\end{proof}
\end{enumerate}

\end{document}